\documentclass{amsart} 
\author{Masaki Kameko}
\address{Department of Mathematical Sciences,
Shibaura Institute of Technology,
307 Minuma-ku Fukasaku, Saitama-City 337-8570, Japan}
\email{kameko@shibaura-it.ac.jp}
\thanks{JSPS KAKENHI Grant Number JP17K05263 supported this work.}

\usepackage{amssymb,amscd, pb-diagram}
\usepackage[initials]{amsrefs}
\newtheorem{theorem}{Theorem}[section]
\newtheorem{proposition}[theorem]{Proposition}

\theoremstyle{definition}

\begin{document}
\title{Coniveau filtrations with $\mathbb{Z}/2$ coefficients} 
\maketitle

\begin{abstract}
We show that two coniveau filtrations on the mod $2$ cohomology group of a smooth projective complex variety differ.
\end{abstract}


\section{Introduction}

Let $X$ be a smooth complex variety of dimension $m$. We say a cohomology class $x\in  H^{i}(X; A)$ with coefficients in an abelian group $A$ has coniveau $\geq c$ if it vanishes outside of a closed subvariety of codimension $\geq c$. We also say that $x$ has strong coniveau $\geq c$ if it is the Gysin pushforward of a cohomology class on a smooth complex variety $Y$ of dimension $\leq m-c$ through a proper morphism $f\colon Y \to X$. We use these notions to define two filtrations on the cohomology group $H^i(X; A)$. We denote them by $N^cH^i(X;A)$, $\tilde{N}^cH^i(X;A)$, respectively. In general, $\tilde{N}^cH^i(X;A)\subset {N}^cH^i(X;A)$. The equality $\tilde{N}^cH^i(X;A)= {N}^cH^i(X;A)$ holds if $X$ is proper and $A=\mathbb{Q}$. Recently, Benoist and Ottem \cite{benoist-ottem-2021} showed that if $X$ is projective and $A=\mathbb{Z}$, the equality does not hold in general. Also, they asked whether the inclusion $\tilde{N}^1 H^{3}(X;\mathbb{Z}/2)\subset N^1H^{3}(X;\mathbb{Z}/2)$ is strict for a smooth projective complex variety $X$. The goal of this paper is to answer this question. Our result is as follows:


\begin{theorem}\label{theorem:1.1} There is a smooth projective complex variety $X$ such that the inclusion $\tilde{N}^1H^{3}(X;\mathbb{Z}/2)\subset N^1H^{3}(X;\mathbb{Z}/2)$ is strict. \end{theorem}

For a compact Lie group $G$, we denote its classifying space by $BG$. By Ekedahl's theorem, for a compact Lie group $G$ and a positive integer $m$, there is a smooth projective complex variety $X$ with an $m$-equivalence $X\to B(G\times S^1)$. Let $PU(n)$ be the projective unitary group that is the quotient group of the unitary group $U(n)$ by its center $S^1$. We show that for $m>10$ and $G=PU(4)$, the generator of $H^3(X;\mathbb{Z}/2)\simeq \mathbb{Z}/2$ has coniveau $\geq 1$ and strong coniveau $<1$. 

For an odd prime $p$ and $A = \mathbb{Z}/p$, it is very likely that replacing $PU(4)$ by $PU(p^2)$ would produce similar examples. To do so, we need to compute the mod $p$ cohomology of $BPU(p^2)$, but even the mod $3$ cohomology groups for $BPU(9)$ are not yet known, suggesting that this would require considerable effort.  This is why we only consider the case $A=\mathbb{Z}/2$ in this paper.

This paper is organized as follows: Section~\ref{sec2} gives a sufficient condition for a cohomology class to have coniveau $\geq 1$. Section~\ref{sec3} provides a sufficient condition for a cohomology class to have strong coniveau $<1$. In Section~\ref{sec4}, we prove the theorem.

In this paper, cohomology groups mean singular cohomology groups of topological spaces. When considering a smooth complex variety as a topological space, we consider the complex topology. So, it is homeomorphic to its underlying $C^\infty$-manifold. We write $H^{*}(X;\mathbb{Z}/2)$ for the mod $2$ cohomology ring of a topological space $X$. By an ideal of $H^{*}(X;\mathbb{Z}/2)$, we mean a homogeneous ideal of the graded ring $H^{*}(X;\mathbb{Z}/2)$. Thus, the quotient ring $H^{*}(X;\mathbb{Z}/2)/I$ is also a graded ring. We write $(H^*(X;\mathbb{Z}/2)/I)^i$ for $H^i(X;\mathbb{Z}/2)/(I\cap H^i(X;\mathbb{Z}/2))$. We denote the $\mathbb{Z}/2$ vector space spanned by a finite set $\{ a, b, \dots, c\}$ by $\mathbb{Z}/2\{a, b, \dots, c\}$. 


\section{A sufficient condition for coniveau $\geq 1$}\label{sec2}

In this section, we recall the mod $2$ reduction map and the definition of coniveau filtration. Then, we state a sufficient condition for a mod $2$ cohomology class to have coniveau $\geq 1$.

\subsection{The mod $2$ reduction map} We begin with the relation between the integral and mod $2$ cohomology groups of a topological space $X$. A short exact sequence \[ \{ 0\}  \to \mathbb{Z} \stackrel{\times 2}{\longrightarrow} \mathbb{Z} \to \mathbb{Z}/2 \to \{ 0\}, \] yields a long exact sequence in cohomology \[ \cdots \to H^i(X;\mathbb{Z})\stackrel{\rho}{\longrightarrow} H^i(X;\mathbb{Z}/2)\stackrel{\delta}{\longrightarrow} H^{i+1}(X;\mathbb{Z})\to\cdots \] where $\rho$ is the mod $2$ reduction map and the composition $\rho\circ \delta$ is known as the mod $2$ Bockstein operation. 

\subsection{Definition of the coniveau filtration} For a smooth complex variety $X$, the coniveau filtration with coefficients in an abelian group $A$ is defined by \[ N^c H^{i}(X;A)=\sum_{Z\subset X} \mathrm{Ker}\, j^*\colon H^{i}(X;A)\to H^{i}(X- Z;A), \] where $Z$ ranges over the closed subvarieties of $X$ of codimension $\geq c$ and $j\colon X-Z\to X$ is the inclusion map.

\subsection{A sufficient condition for coniveau $\geq 1$} For $A=\mathbb{Z}$, it is known that any torsion integral cohomology class of a smooth complex variety has coniveau $\geq 1$. See \cite[Proposition 2.8]{benoist-ottem-2021}. An immediate consequence of this fact is the following proposition.


\begin{proposition}\label{proposition:2.1} Let $X$ be a smooth complex variety. Suppose that $H^i(X;\mathbb{Z})$ is a cyclic group generated by a cohomology class $x$ and $\rho(x)\not=0$ in $H^i(X;\mathbb{Z}/2)$. Then, $\rho(x)$ has coniveau $\geq 1$. \end{proposition}

\begin{proof} Since $H^i(X;\mathbb{Z})$ is a cyclic group, there is a closed subvariety $Z$ of codimension $\geq 1$ such that \[  x\in \mathrm{Ker}\, j^*\colon H^i(X;\mathbb{Z})\to H^{i}(X-Z;\mathbb{Z}), \] where $j\colon X-Z\to X$ is the inclusion map. Consider the following commutative diagram: \[ \begin{diagram} \node{H^{*}(X;\mathbb{Z})} \arrow{s,l}{\rho}\arrow{e,t}{j^*} \node{H^{*}(X- Z; \mathbb{Z})}\arrow{s,r}{\rho} \\ \node{H^{*}(X;\mathbb{Z}/2)} \arrow{e,t}{j^*}\node{H^{*}(X- Z;\mathbb{Z}/2).} \end{diagram} \] Then, we have $j^*\circ \rho(x)=\rho\circ j^{*}(x)=\rho(0)=0$. Hence, we have $\rho(x)\in N^1H^{i}(X;\mathbb{Z}/2)$. \end{proof}.


\section{A sufficient condition for strong coniveau $<1$}\label{sec3} In this section, firstly, we recall elementary properties of Steenrod operations. Secondly, we recall the definition and elementary properties of Milnor operations. Thirdly, we recall the definition of the Gysin pushforward map for a proper morphism of smooth complex varieties and show that Milnor operations commute with the Gysin pushforward map. Finally, we recall the definition of the strong coniveau filtration for a smooth complex variety and state a sufficient condition for a mod $2$ cohomology class to have strong coniveau $<1$.


\subsection{Steenrod operations} We begin with elementary properties of Steenrod operations. Let $X$ be a topological space. Steenrod operations $\mathrm{Sq}^i: H^j(X;\mathbb{Z}/2)\to H^{j+i}(X;\mathbb{Z}/2)$ are natural transformations of the mod $2$ cohomology groups $H^{i}(X;\mathbb{Z}/2)$ as functors from the category of topological spaces to that of graded $\mathbb{Z}/2$ modules. They satisfy the unstable condition \begin{align*} \mathrm{Sq}^i (x)&=x^2\quad \mbox{if $\deg x=i$, and} \\ \mathrm{Sq}^i (x)&=0 \quad\mbox{if $\deg x<i,$} \end{align*} and the Cartan formula \begin{align*} \mathrm{Sq}^i (x \cdot y)&=\sum_{j=0}^{i} \mathrm{Sq}^{i-j}(x) \cdot \mathrm{Sq}^j(y) \end{align*} where $\mathrm{Sq}^0$ is the identity map. In addition, we use the Adem relation $\mathrm{Sq}^1 \mathrm{Sq}^2=\mathrm{Sq}^3$.

\subsection{Milnor operations} Next, we recall the definition of Milnor operations $$Q_i\colon H^j(X;\mathbb{Z}/2)\to H^{j+2^{i+1}-1}(X;\mathbb{Z}/2).$$ Milnor operations $Q_i$ are defined inductively by $Q_0:=\mathrm{Sq}^1$, and $Q_i:=\mathrm{Sq}^{2^i} Q_{i-1}+Q_{i-1} \mathrm{Sq}^{2^i}$ for $i\geq 1$. It is well-known that $Q_i$ acts on the mod $2$ cohomology ring $H^{*}(X;\mathbb{Z}/2)$ as a derivation, so that we have \[ Q_i(x\cdot y)=Q_i(x)\cdot y+ x \cdot Q_i(y) \] for $x, y$ in $H^{*}(X;\mathbb{Z}/2)$. See Remark after Lemma 9 in \cite{milnor-1958}. We use the following property of Milnor operations in the proof of Proposition~\ref{proposition:3.3}.


\begin{proposition}\label{proposition:3.1} Suppose that $x_1\in H^{1}(X;\mathbb{Z}/2)$. Then, we have $Q_i (x_1^{2j+1})=x_1^{2j+2^{i+1}}$. \end{proposition}

\begin{proof} Since $Q_i$ acts as a derivation, we have \begin{align*} Q_i(x_1\cdot x_1^j \cdot x_1^j)&=Q_i(x_1) \cdot x_1^j \cdot x_1^j +x_1 \cdot Q_i(x_1^j) \cdot x_1^j + x_1\cdot x_1^j \cdot Q_i(x_1^j) \\ & =Q_i(x_1)\cdot x_1^{j}\cdot x_1^j \\ &= Q_i(x) \cdot x_1^{2j}. \end{align*} We show that $Q_i(x_1)=x_1^{2^{i+1}}$ by induction on $i$. For $i=0$, by the definition of $Q_0$ and the unstable condition $\mathrm{Sq}^1 (x_1)=x_1^2$, we have \[ Q_0(x_1)=\mathrm{Sq}^1 (x_1)=x_1^2. \] For $i\geq 1$, by the definition of the Milnor operation $Q_i$, the unstable condition $\mathrm{Sq}^{2^{i}}(x_1)=0$, the inductive hypothesis $Q_{i-1}(x_1)=x_1^{2^{i}}$ and the unstable condition $\mathrm{Sq}^{2^i}(x_1^{2^i})=x_1^{2^{i+1}}$, we have \[ Q_i(x_1)=(\mathrm{Sq}^{2^i} Q_{i-1}+Q_{i-1}\mathrm{Sq}^{2^i}) (x_1)=\mathrm{Sq}^{2^i} Q_{i-1} (x_1)=\mathrm{Sq}^{2^i} (x_1^{2^i})=x_1^{2^{i+1}}. \] Therefore, for all $i\geq 0$, we have $Q_1(x_1)=x_1^{2^{i+1}}$. Combining this with $Q_i(x_1^{2j+1})=Q_i(x_1) \cdot x_1^{2j}$, we obtain the desired result. \end{proof}


\subsection{The Gysin pushforward map} We recall the definition of the Gysin pushforward map. Let $X$ be a $C^\infty$-manifold. Let  $\xi \to X$ be a $2k$-dimensional vector bundle over $X$. We denote the total and base spaces by $\xi$, $X$, respectively. We denote by $D(\xi)$, $S(\xi)$ the total spaces of associated disk and sphere bundles so that $S(\xi)\subset D(\xi)$. Suppose $\xi$ has an almost complex structure. Then, we have a classifying map, unique up to homotopy, $c(\xi)\colon X\to BU(k)$  so that $\xi$ is isomorphic to the pull-back of the universal $k$-dimensional complex vector bundle $\xi_k\to BU(k)$ by $c(\xi)$. We write $u_{2k}(\xi)$ for the Thom class of $\xi$. Then, $u_{2k}(\xi)=c(\xi)^{*}(u_{2k}(\xi_k))$. We denote by $t(\xi)\colon H^i(X;\mathbb{Z}/2) \to H^{i+2k}(D(\xi), S(\xi);\mathbb{Z}/2)$ the Thom isomorphism. The Thom isomorphism is given by the multiplication by the Thom class, $x\mapsto x \cdot u_{2k}(\xi)$.

Let $\mu=X \times \mathbb{R}^{2n}$. We regard the projection map to the first factor $\mu\to X$ as a trivial $2n$-dimensional vector bundle over $X$. By Whitney's embedding theorem, every $C^\infty$-manifold $Y$ admits a closed embedding $\iota \colon Y\to \mathbb{R}^{2n}$, see e.g. \cite[Theorem 2.6]{adachi-1993}, \cite[Theorem 6.15]{lee-2013}. Moreover, if $n$ is sufficiently large, it is unique up to isotopy. Therefore, if $f\colon Y\to X$ is a proper $C^\infty$-map, $f$ factors through a closed embedding $(f,\iota)\colon Y \to \mu$ and if $n$ is sufficiently large, the closed embedding $(f,\iota)$ is unique up to isotopy. Let $\nu \to Y$ be a normal bundle of the closed embedding $(f,\iota)$. By taking a tubular neighborhood of $Y$ in $\mu$, we identify $D(\nu)$ as a closed subset of $\mu$. Since $\mu$ is diffeomorphic to the interior of $D(\mu)$, $\mathrm{int}\, D(\mu)=D(\mu)-S(\mu)$, we may assume that $D(\nu)\subset D(\mu)$ without loss of generality.  Thus, we have the following inclusion maps.
\[ (D(\nu), S(\nu)) \stackrel{j(\nu)}{\longrightarrow} (\mu, \mu- \mathrm{int}\, D(\nu)) \stackrel{j(\nu, \mu)}{\longleftarrow} (\mu, \mu- \mathrm{int}\, D(\mu))\stackrel{j(\mu)}{\longleftarrow} (D(\mu), S(\mu)). \] By the excision theorem, the induced homomorphisms $j(\mu)^*$, $j(\nu)^{*}$ are isomorphisms. 

Suppose that $f\colon Y\to X$ is a proper morphism of smooth complex varieties and $n$ is sufficiently large. Then, the normal bundle $\nu\to Y$ of the closed embedding $(f,\iota)$ has an almost complex structure. Let $j=\dim X-\dim Y$ where $\dim$ is the dimension of a smooth complex variety, not the dimension of the underlying $C^\infty$-manifold. Let $c(\nu)\colon Y\to BU(j+n)$ be the classifying map of the normal bundle $\nu$. We may define the Gysin pushforward map $f_*\colon H^{i-2j}(Y;\mathbb{Z}/2)\to H^{i}(X;\mathbb{Z}/2)$ as the composition \[ t(\mu)^{-1}\circ j(\mu)^* \circ j(\nu,\mu)^* \circ (j(\nu)^{*})^{-1} \circ t(\nu). \]

The following fact is known. See the comment after Lemma 2.1 in \cite{yagita-2023}. For the reader's convenience, we prove it here.


\begin{proposition}\label{proposition:3.2} Let $X$ and $Y$ be smooth complex varieties. Let $j=\dim X-\dim Y$. For a proper morphism of smooth complex varieties $f\colon Y \to X$, Milnor operations $Q_i$ commute with the Gysin pushforward map $f_*:H^{i-2j}(Y;\mathbb{Z}/2)\to H^{i}(X;\mathbb{Z}/2)$. \end{proposition}

\begin{proof} By definition, $f_*= t(\mu)^{-1}\circ j(\mu)^* \circ j(\nu,\mu)^* \circ {j(\nu)^*}^{-1}\circ t(\nu)$. Since $\mu\to X$ is a trivial vector bundle, the Thom isomorphism $t(\mu)$ commutes with Steenrod operations. Thus, all homomorphisms in the above composition commute with Steenrod operations except for the Thom isomorphism $t(\nu)$. Hence, they commute with Milnor operations. We show that $t(\nu)$ also commutes with Milnor operations.

The Thom isomorphism $t(\nu)$ is defined as the multiplication by the Thom class $u_{2j+2n}(\nu)=c(\nu)^*(u_{2j+2n}(\xi_{j+n}))$. Since \[ H^{2k+1}(D(\xi_{j+n}), S(\xi_{j+n});\mathbb{Z}/2)\simeq H^{2k+1-2(j+n)}(BU(j+n);\mathbb{Z}/2)\simeq  \{0\} \] for all $k\geq 0$, Milnor operations $Q_i$ act trivially on $u_{2j+2n}(\xi_{j+n})$. Hence, we have \begin{align*} Q_i(y \cdot u_{2j+2n}(\nu))&=Q_i(y) \cdot u_{2j+2n}(\nu)+y \cdot Q_i(u_{2j+2n}(\nu)) \\ &=Q_i(y) \cdot u_{2j+2n}(\nu)+y \cdot Q_i(c(\nu)^*(u_{2j+2n}(\xi_{j+n}))) \\ &=Q_i(y) \cdot u_{2j+2n}(\nu)+y \cdot c(\nu)^*(Q_i(u_{2j+2n}(\xi_{j+n}))) \\ &=Q_i(y) \cdot u_{2j+2n}(\nu). \end{align*} Therefore, we have \[ Q_i  \circ t(\nu)=t(\nu)\circ Q_i \]and we obtain the desired result. \end{proof}


\subsection{Definition of strong coniveau filtration} The strong coniveau filtration $\tilde{N}^cH^i(X;A)$ is defined as follow: For an $m$-dimensional smooth complex variety $X$, \[ \tilde{N}^cH^i(X;A)=\sum_{f\colon Y\to X} \mathrm{Im}\, f_*\colon H^{i-2j}(Y;A)\to H^i(X;A) \] where $f$ ranges over the proper morphisms with $j=m-\dim Y\geq c$.


\subsection{A sufficient condition for strong coniveau $<1$} Let us consider an ideal $I$ of the mod $2$ cohomology ring $H^{*}(X;\mathbb{Z}/2)$ invariant under the action of Milnor operations. Then, Milnor operations act on the quotient ring $H^{*}(X;\mathbb{Z}/2)/I$. 


\begin{proposition}\label{proposition:3.3} Let $X$ be a smooth complex variety. Let $I$ be an ideal of the mod $2$ cohomology ring $H^{*}(X;\mathbb{Z}/2)$ invariant under the action of Milnor operations. Suppose that $H^3(X;\mathbb{Z}/2)\simeq \mathbb{Z}/2$ and generated by $x_3$. If $Q_2(x_3)\not= 0$ in $H^{*}(X;\mathbb{Z}/2)/I$ and if $(H^*(X;\mathbb{Z}/2)/I)^{7}=\{0\},$ then, $x_3$ has coniveau $<1$. \end{proposition}

\begin{proof} Suppose that $\dim X=m$. Let $Y$ be an $(m-j)$-dimensional smooth complex variety and $f\colon Y \to X$ be a proper morphism such that $j\geq 1$. We show that $x_3\not=0 \in H^{3}(X;\mathbb{Z}/2)$ is not in the image of the Gysin pushforward map $f_*$. The Gysin pushforward map is of degree $2j$, so that $f_*\colon H^{3-2j}(Y;\mathbb{Z}/2)\to H^3(X;\mathbb{Z}/2)$. For $j>1$, $3-2j<0$, and so the image of the Gysin pushforward map is trivial. Hence, we only need to deal with the case $j=1$. Suppose $j=1$. Then, the Gysin pushforward map is $f_*\colon H^1(Y;\mathbb{Z}/2)\to H^3(X;\mathbb{Z}/2)$. Let $y_1$ be an element in $H^1(Y;\mathbb{Z}/2)$. Suppose that $f_*(y_1)=\alpha x_3$ for some scalar $\alpha\in \mathbb{Z}/2$. We show that $\alpha=0$.

By Proposition~\ref{proposition:3.1}, we have $Q_2(y_1)=Q_1(y_1^5)=y_1^8$. Applying the Gysin pushforward map and using Proposition~\ref{proposition:3.2}, we have $Q_2\circ  f_*(y_1)= Q_1 \circ f_*(y_1^5)$ in $H^{10}(X;\mathbb{Z}/2)$. On the one hand, we have $Q_2 \circ f_*(y_1)= \alpha Q_2 (x_3)$ in $H^{*}(X;\mathbb{Z}/2)/I$. On the other hand, since $(H^*(X;\mathbb{Z}/2)/I)^7=\{0\}$, we have $Q_1\circ  f_*(y_1^5)= 0$ in $H^*(X;\mathbb{Z}/2)/I$. Therefore, we have $\alpha Q_2(x_3)=0$ in $H^{*}(X;\mathbb{Z}/2)/I$. Since $Q_2(x_3)\not=0$ in $H^{*}(X;\mathbb{Z}/2)/I$, we obtain $\alpha=0$. \end{proof}


\section{Proof of Theorem~\ref{theorem:1.1}}\label{sec4} In this section, first, we recall Ekedahl's theorem on smooth projective complex varieties approximating the classifying space of a compact Lie group. Next, we recall the mod $2$ cohomology ring of the classifying space of the compact Lie group $PU(4)$. Then, we show the mod $2$ cohomology of the smooth projective complex variety $X$ with an $m$-equivalence $X\to B(PU(4)\times S^1)$ satisfies the conditions in Propositions~\ref{proposition:2.1} and \ref{proposition:3.3} when $m>10$ to complete the proof of Theorem~\ref{theorem:1.1}.


\subsection{Approximation of classifying space} The complexification $G_{\mathbb{C}}$ of a connected compact Lie group $G$ is a connected reductive complex Lie group. The Lie group $G$ naturally embeds in its complexification as a real Lie subgroup, and is a maximal compact subgroup of $G_\mathbb{C}$. Furthermore, the inclusion map $G\to G_\mathbb{C}$ is a homotopy equivalence. Thus, regarding the cohomology of classifying spaces, there is no difference between compact and reductive complex Lie groups. Approximations of the classifying space of a reductive complex linear algebraic group by smooth projective complex varieties are first obtained by  Ekedahl \cite{ekedahl-2009} and improved by Pirutka and Yagita \cite{pirutka-yagita-2015}, Antieau \cite{antieau-2016}, and Benoist \cite{benoist-2025}. Replacing, in Proposition 4.4 of \cite{benoist-2025}, a connected reductive complex Lie group and $\mathbb{C}^*=\mathbb{C}- \{0\}$ by a connected compact Lie group and $S^1$, respectively,  we have the following proposition.


\begin{proposition}\label{proposition:4.1} Let $G$ be a connected compact Lie group and $m$ be a positive integer. Then a smooth projective complex variety $X$ of dimension $m$ with an $m$-equivalence $g\colon X\to B(G \times S^1)$ exists.\end{proposition}

To consider the ideal $I$ of $H^{*}(X;\mathbb{Z}/2)$ in Proposition~\ref{proposition:3.3} for the $m$-equivalence $g\colon X\to B(G\times S^1)$ in Proposition~\ref{proposition:4.1}, let $\mathrm{in}_1 \colon G\to G\times S^1$ be the inclusion map $\mathrm{in}_1(g)=(g, 1)$ where $1\in S^1\subset \mathbb{C}$ and $\mathrm{pr}_1\colon G\times S^1\to G$ be the projection map to the first factor. Let $I$ be the ideal of $H^{*}(X;\mathbb{Z}/2)$ generated by $g^*(\mathrm{Ker}\, B\mathrm{in}_1^*)$. $I$ is invariant under the action of Steenrod operations. Therefore, it is invariant under the action of Milnor operations. Furthermore, throughout the rest of this paper, we assume that $m>10$. Since $g$ is an $m$-equivalence, and since $m>10$, the induced homomorphism $g^*\colon H^i(B(G\times S^1);A)\to H^i(X;A)$ is an isomorphism for $i\leq 10$ and any abelian group $A$.


\begin{proposition}\label{proposition:4.2} With the notations and assumptions above, the following holds{\rm :} \begin{itemize} \item[{\rm (1)}] The composition $g^*\circ B\mathrm{pr}_1^*\colon H^3(BG;\mathbb{Z})\to H^3(X;\mathbb{Z})$ is an isomorphism. Similarly, in mod $2$ cohomology, the composition $g^*\circ B\mathrm{pr}_1^*\colon H^3(BG;\mathbb{Z}/2)\to H^3(X;\mathbb{Z}/2)$ is also an isomorphism. \item[{\rm (2)}] The composition $g^*\circ B\mathrm{pr}_1^*\colon H^{i}(BG;\mathbb{Z}/2)\to  (H^*(X;\mathbb{Z}/2)/I)^i$ is an isomorphism for $i\leq 10$. \end{itemize} \end{proposition}

\begin{proof} (1) There is a homotopy equivalence $B(G\times S^1)\simeq BG \times BS^1$. The integral cohomology ring of $BS^1$ is a polynomial ring generated by a degree $2$ generator over the integers $\mathbb{Z}$. Therefore, by the homotopy equivalence $B(G\times S^1)\simeq BG \times BS^1$ and the K\"{u}nneth theorem, we have \[ H^{*}(B(G\times S^1);\mathbb{Z})\simeq H^{*}(BG \times BS^1;\mathbb{Z})\simeq H^{*}(BG;\mathbb{Z})\otimes H^{*}(BS^1;\mathbb{Z}). \] Since $H^{2i+1}(BS^1;\mathbb{Z})=\{0\}$ for $i\geq 0$, we have \[ H^3(B(G\times S^1);\mathbb{Z})\simeq H^{1}(BG;\mathbb{Z}) \otimes H^2(BS^1;\mathbb{Z}) \oplus H^{3}(BG;\mathbb{Z}) \otimes H^0(BS^1;\mathbb{Z}). \] Moreover, since $BG$ is simply-connected, $H^1(BG;\mathbb{Z})=\{0\}$. Thus, we have \[ H^3(B(G\times S^1);\mathbb{Z})\simeq H^{3}(BG;\mathbb{Z})\otimes H^0(BS^1;\mathbb{Z})\simeq H^3(BG;\mathbb{Z}) \] and we obtain the desired result for integral cohomology groups. Replacing the coefficient group $\mathbb{Z}$ by $\mathbb{Z}/2$, we have the same result for the mod $2$ cohomology groups. \\ (2) Since the composition $\mathrm{pr}_1\circ \mathrm{in}_1$ is the identity map, we have an isomorphism \[ B\mathrm{in}_1^*\colon H^{*}(B(G\times S^1);\mathbb{Z}/2)/(\mathrm{Ker}\, B\mathrm{in}_1^*)\simeq H^{*}(BG;\mathbb{Z}/2). \] Therefore, the induced homomorphism \[ B\mathrm{pr}_1^*\colon H^*(BG;\mathbb{Z}/2)\to H^{*}(B(G\times S^1);\mathbb{Z}/2)/(\mathrm{Ker}\, B\mathrm{in}_1^*) \] is an isomorphism. It is also clear that the induced homomorphism \[ g^*\colon (H^{*}(B(G\times S^1);\mathbb{Z}/2)/(\mathrm{Ker}\, B\mathrm{in}_1^*))^i \to (H^{*}(X;\mathbb{Z}/2)/I)^i \] is an isomorphism for $i\leq 10$.
\end{proof}


\subsection{Cohomology of $BPU(4)$} Kono and Mimura \cite{kono-mimura-1974} computed the mod $2$ cohomology ring of the classifying space of the projective special orthogonal group $PSO(4i+2)$ for $i\geq 1$. The projective unitary group $PU(4)$ is the quotient group of the special unitary group $SU(4)$ by its center $\mathbb{Z}/4$.  The projective special orthogonal group  $PSO(6)$ is the quotient group of the spin group $\mathrm{Spin}(6)$ by its center $\mathbb{Z}/4$. The special unitary group $SU(4)$ is isomorphic to the spin group $\mathrm{Spin}(6)$. Therefore, $PU(4)$ is isomorphic to $PSO(6)$ and we have the mod $2$ cohomology ring of $BPU(4)$ from that of $BPSO(6)$. We refer the reader to (4.10) on page 103 in Toda \cite{toda-1984} for the following result.


\begin{proposition}\label{proposition:4.3} The mod $2$ cohomology ring of the  classifying space $BPU(4)$ is generated by $x_2$, $x_3$, $x_5$, $x_8$, $x_9$, $x_{12}$ with the relations $x_2 x_3=x_2 x_5=x_2 x_9=x_9^2+x_3^2 x_{12}+x_5^2x_8=0$, where subscripts indicate degrees. \end{proposition}

Thus, we have the mod $2$ cohomology ring of $BPU(4)$ as follows: \[ H^{*}(BPU(4);\mathbb{Z}/2)=\mathbb{Z}/2[ x_2, x_3, x_5, x_8, x_{9}, x_{12}]/( x_2 x_3, x_2 x_5, x_2 x_9, x_9^2+x_3^2 x_{12}+x_5^2x_8). \] However, what we use in this paper is only the following: \begin{align*} H^3(BPU(4);\mathbb{Z}/2)&=\mathbb{Z}/2\{ x_3\},\\ H^5(BPU(4);\mathbb{Z}/2)&=\mathbb{Z}/2\{ x_5\}, \\ H^7(BPU(4);\mathbb{Z}/2)&=\{ 0 \} \end{align*} and $x_3^2\not=0$, $x_5^2\not=0$ in $H^{6}(BPU(4);\mathbb{Z}/2)$, $H^{10}(BPU(4);\mathbb{Z}/2)$, respectively.


\begin{proposition}\label{proposition:4.4} The generator $x_3$ of $H^3(BPU(4);\mathbb{Z}/2)\simeq \mathbb{Z}/2$ is the mod $2$ reduction of a generator of the integral cohomology group $H^3(BPU(4);\mathbb{Z})\simeq \mathbb{Z}/4$. \end{proposition}

\begin{proof} Since $PU(4)$ is connected, we have $\pi_1(BPU(4))=\{0\}$. Considering the fiber sequence $B\mathbb{Z}/4 \to BSU(4) \to BPU(4), $ and taking account of the fact that $\pi_2(BSU(4))\simeq \{0\}$, we have $\pi_2(BPU(4))\simeq \mathbb{Z}/4.$ By the Hurewicz theorem, we have $H_2(BPU(4);\mathbb{Z})\simeq \pi_2(BPU(4))\simeq \mathbb{Z}/4$. Moreover, computing the rational homology groups \[ H_3(BPU(4);\mathbb{Q})\simeq H_3(BSU(4);\mathbb{Q})\simeq \{0\}, \] we have that the integral cohomology group $H_3(BPU(4);\mathbb{Z})$ is a torsion group and \[ \mathrm{Hom}_\mathbb{Z}(H_3(BPU(4);\mathbb{Z}), \mathbb{Z})\simeq \{0\}. \] By the universal coefficient theorem for the integral cohomology group, we obtain \[ H^3(BPU(4);\mathbb{Z})\simeq \mathrm{Ext}_\mathbb{Z}^1 (H_2(BPU(4);\mathbb{Z}), \mathbb{Z})\simeq \mathbb{Z}/4. \] By the universal coefficient theorem for the mod $2$ cohomology group, the mod $2$ cohomology group $H^3(BPU(4);\mathbb{Z}/2)$ has a direct summand \[ \mathrm{Ext}_{\mathbb{Z}}^1(H_2(BPU(4);\mathbb{Z}), \mathbb{Z}/2)\simeq \mathbb{Z}/2. \] On the other hand, by Proposition~\ref{proposition:4.3}, the dimension of $H^3(BPU(4);\mathbb{Z}/2)$ is $1$. Therefore, we have \[ H^3(BPU(4);\mathbb{Z}/2)\simeq \mathrm{Ext}_\mathbb{Z}^1 (H_2(BPU(4);\mathbb{Z}), \mathbb{Z}/2)\simeq \mathbb{Z}/2. \] The mod $2$ reduction $\mathbb{Z} \to \mathbb{Z}/2$ of coefficients yields a surjective homomorphism \[ \mathrm{Ext}_\mathbb{Z}^1 (H_2(BPU(4);\mathbb{Z}),\mathbb{Z})\to \mathrm{Ext}_\mathbb{Z}^1(H_2(BPU(4);\mathbb{Z}),\mathbb{Z}/2). \] It is equivalent to the mod $2$ reduction map $\rho\colon H^{3}(BPU(4);\mathbb{Z})\to H^3(BPU(4);\mathbb{Z}/2)$. Therefore, the generator $x_3$ of $H^{3}(BPU(4);\mathbb{Z}/2)$ is the mod $2$ reduction of a generator of $H^3(BPU(4);\mathbb{Z})\simeq \mathbb{Z}/4$. \end{proof}


\begin{proposition} \label{proposition:4.5} The following holds for $H^{*}(BPU(4);\mathbb{Z}/2)${\rm :} \begin{itemize} \item[{\rm (1)}] $Q_1 (x_3)=Q_0 \mathrm{Sq}^2 (x_3)=x_3^2\not=0$, \item[{\rm (2)}] $\mathrm{Sq}^2 (x_3)=x_5$, and \item[{\rm (3)}] $Q_2 (x_3)=x_5^2\not=0$. \end{itemize} \end{proposition}

\begin{proof} (1) By Proposition~\ref{proposition:4.4}, $x_3$ is the mod $2$ reduction of an integral cohomology class $x_3'$. Hence, we have $Q_0 (x_3)=\rho\circ \delta\circ \rho(x_3')=0$. By the definition of Milnor operation $Q_1$, $Q_0(x_3)=0$,  the definition of the Milnor operation $Q_0$, the Adem relation $\mathrm{Sq}^1 \mathrm{Sq}^2=\mathrm{Sq}^3$, and the unstable condition $\mathrm{Sq}^3 (x_3)=x_3^2$, we have \[ Q_1 (x_3)=(\mathrm{Sq}^2 Q_0+Q_0 \mathrm{Sq}^2) (x_3)= Q_0 \mathrm{Sq}^2 (x_3)=\mathrm{Sq}^1 \mathrm{Sq}^2 (x_3)=\mathrm{Sq}^3 (x_3)=x_3^2. \] By Proposition~\ref{proposition:4.3}, we have $x_3^2\not=0$. \\ (2) Since $Q_0 \mathrm{Sq}^2 (x_3)\not=0$,  we have $\mathrm{Sq}^2 (x_3)\not=0.$ By Proposition~\ref{proposition:4.3}, there is only one nonzero element $x_5$ in $H^{5}(BPU(4);\mathbb{Z}/2)$. Therefore, we have $\mathrm{Sq}^2 (x_3)=x_5$. \\ (3) By the definition of $Q_2$, the unstable condition $\mathrm{Sq}^4(x_3)=0$, $Q_1(x_3)=x_3^2$, the Cartan formula, and $\mathrm{Sq}^2 (x_3)=x_5$, we have \[ Q_2(x_3)=(\mathrm{Sq}^4 Q_1+ Q_1 \mathrm{Sq}^4) (x_3)=\mathrm{Sq}^4 Q_1 (x_3)=\mathrm{Sq}^4 (x_3^2)=(\mathrm{Sq}^2 (x_3))^2=x_5^2. \] By Proposition~\ref{proposition:4.3}, we have $x_5^2\not=0$. \end{proof}


\subsection{Proof of Theorem~\ref{theorem:1.1}} Finally, we prove Theorem~\ref{theorem:1.1}.


\begin{proof}[Proof of Theorem~\ref{theorem:1.1}] Suppose that $m>10$. By Proposition~\ref{proposition:4.1}, there is an $m$-dimensional  smooth projective complex variety $X$ with an $m$-equivalence $g\colon X\to B(PU(4)\times S^1)$. Let $\mathrm{pr}_1\colon PU(4)\to S^1$ be the projection map to the first factor and $\mathrm{in}_1\colon PU(4) \to PU(4)\times S^1$ be the inclusion map $g\mapsto (g, 1)$. We define $I$ as the ideal of $H^{*}(X;\mathbb{Z}/2)$ generated by $g^*(\mathrm{Ker}\, B\mathrm{in}_1^*)$. 

By Proposition~\ref{proposition:4.4},  we have $H^{3}(BPU(4);\mathbb{Z})\simeq \mathbb{Z}/4$, $H^3(BPU(4);\mathbb{Z}/2)\simeq \mathbb{Z}/2$ and if $x_3'$ is a generator of the integral cohomology group $H^3(BPU(4);\mathbb{Z})$, then $x_3=\rho (x'_3)$. By Proposition~\ref{proposition:4.2} (1), $H^3(X;\mathbb{Z})\simeq \mathbb{Z}/4$, $H^3(X;\mathbb{Z}/2)\simeq \mathbb{Z}/2$, $g^{*}\circ B\mathrm{pr}_1^*(x_3')$ is a generator of $H^3(X;\mathbb{Z})$ and its mod $2$ reduction $\rho \circ g^{*}\circ B\mathrm{pr}_1^*(x_3')=g^{*}\circ B\mathrm{pr}_1^*(x_3)$ is also the generator of $H^3(X;\mathbb{Z}/2)$. Therefore, by Proposition~\ref{proposition:2.1}, we have $g^{*}\circ B\mathrm{pr}_1^*(x_3) \in N^1H^3(X;\mathbb{Z}/2)$.

By Propositions \ref{proposition:4.2} (2) and \ref{proposition:4.5} (3), we have\[ Q_2(g^* \circ B\mathrm{pr}_1^*(x_3))=g^* \circ B\mathrm{pr}_1^*(x_5^2)\not=0 \] in \[ (H^{*}(X;\mathbb{Z}/2)/I)^{10}\simeq H^{10}(BPU(4);\mathbb{Z}/2). \] By Propositions~\ref{proposition:4.2} (2) and \ref{proposition:4.3}, we have \[ (H^{*}(X;\mathbb{Z}/2)/I)^7\simeq H^{7}(BPU(4);\mathbb{Z}/2)\simeq \{0\}. \] Hence, by Proposition~\ref{proposition:3.3}, we have $g^{*}\circ B\mathrm{pr}_1^*(x_3)\not\in \tilde{N}^1H^3(X;\mathbb{Z}/2). $

Thus, the inclusion $\tilde{N}^1H^3(X;\mathbb{Z}/2)\subset N^1H^3(X;\mathbb{Z}/2)$ is strict. \end{proof}



\BibSpec{arXiv}{%
  +{}{\PrintAuthors}{author}
  +{,}{ \textit}{title}
  +{}{ \parenthesize}{date}
  +{,}{ arXiv:}{eprint}
}



\begin{bibdiv}[References]
\begin{biblist}


\bib{adachi-1993}{book}{
   author={Adachi, Masahisa},
   title={Embeddings and immersions},
   series={Translations of Mathematical Monographs},
   volume={124},
   note={Translated from the 1984 Japanese original by Kiki Hudson},
   publisher={American Mathematical Society, Providence, RI},
   date={1993},
   pages={x+183},
   isbn={0-8218-4612-4},
   doi={10.1090/mmono/124},
}

\bib{antieau-2016}{article}{
   author={Antieau, Benjamin},
   title={On the integral Tate conjecture for finite fields and
   representation theory},
   journal={Algebr. Geom.},
   volume={3},
   date={2016},
   number={2},
   pages={138--149},
   issn={2313-1691},
   doi={10.14231/AG-2016-007},
}

\bib{benoist-2025}{article}{
   author={Benoist, Olivier},
   title={Steenrod operations and algebraic classes},
   journal={Tunis. J. Math.},
   volume={7},
   date={2025},
   number={1},
   pages={53--89},
   issn={2576-7658},
   doi={10.2140/tunis.2025.7.53},
}

\bib{benoist-ottem-2021}{article}{
   author={Benoist, Olivier},
   author={Ottem, John Christian},
   title={Two coniveau filtrations},
   journal={Duke Math. J.},
   volume={170},
   date={2021},
   number={12},
   pages={2719--2753},
   issn={0012-7094},
   doi={10.1215/00127094-2021-0055},
}

\bib{ekedahl-2009}{arXiv}{
author={Ekedahl, Torsten}, 
title={Approximating classifying spaces by smooth projective varieties}, 
eprint={0905.1538},
}


\bib{kono-mimura-1974}{article}{
   author={Kono, Akira},
   author={Mimura, Mamoru},
   title={On the cohomology of the classifying spaces of ${\rm PSU}(4n+2)$
   and ${\rm PO}(4n+2)$},
   journal={Publ. Res. Inst. Math. Sci.},
   volume={10},
   date={1974/75},
   number={3},
   pages={691--720},
   issn={0034-5318},
   doi={10.2977/prims/1195191887},
}

\bib{lee-2013}{book}{
   author={Lee, John M.},
   title={Introduction to smooth manifolds},
   series={Graduate Texts in Mathematics},
   volume={218},
   edition={2},
   publisher={Springer, New York},
   date={2013},
   pages={xvi+708},
   isbn={978-1-4419-9981-8},
}

\bib{milnor-1958}{article}{
   author={Milnor, John},
   title={The Steenrod algebra and its dual},
   journal={Ann. of Math. (2)},
   volume={67},
   date={1958},
   pages={150--171},
   issn={0003-486X},
   doi={10.2307/1969932},
}


\bib{pirutka-yagita-2015}{article}{
   author={Pirutka, Alena},
   author={Yagita, Nobuaki},
   title={Note on the counterexamples for the integral Tate conjecture over
   finite fields},
   journal={Doc. Math.},
   date={2015},
   pages={501--511},
   issn={1431-0635},
}

\bib{toda-1984}{article}{
   author={Toda, Hiroshi},
   title={Cohomology of classifying spaces},
   conference={
      title={Homotopy theory and related topics},
      address={Kyoto},
      date={1984},
   },
   book={
      series={Adv. Stud. Pure Math.},
      volume={9},
      publisher={North-Holland, Amsterdam},
   },
   isbn={0-444-70201-6},
   date={1987},
   pages={75--108},
   doi={10.2969/aspm/00910075},
}
\bib{yagita-2023}{article}{
   author={Yagita, Nobuaki},
   title={Coniveau filtrations and Milnor operation $Q_n$},
   journal={Math. Proc. Cambridge Philos. Soc.},
   volume={175},
   date={2023},
   number={3},
   pages={521--538},
   issn={0305-0041},
   doi={10.1017/s0305004123000282},
}
\end{biblist}
\end{bibdiv}
\end{document}